\documentclass[reqno]{amsart}

\usepackage[utf8]{inputenc}
\usepackage{microtype}
\usepackage{amsfonts}
\usepackage{amsmath}
\usepackage{comment}
\usepackage{graphicx}
\usepackage{hyperref}
\usepackage{color}
\usepackage{float}

\usepackage{amssymb}
\usepackage{enumitem}

\usepackage{tikz}
\usetikzlibrary{topaths}
\usetikzlibrary{calc}

\usepackage{a4wide}

\usepackage{empheq}

\newtheorem{theorem}{Theorem}[section]
\newtheorem*{theorem*}{Theorem}

\newtheorem{lemma}[theorem]{Lemma}

\newtheorem{proposition}[theorem]{Proposition}
\newtheorem*{proposition*}{Proposition}
\newtheorem{corollary}[theorem]{Corollary}
\newtheorem*{corollary*}{Corollary}
\newtheorem{conjecture}[theorem]{Conjecture}

\newtheorem*{conjecture*}{Conjecture}

\newtheorem*{question*}{Question}

\newtheorem*{main:main_1_ground_morse}{Theorem~\ref{thrm:1_ground_morse}}
\newtheorem*{main:main_graphs}{Theorem~\ref{thrm:main_graphs}}
\newtheorem*{main:main_classification}{Theorem~\ref{thrm:classification}}

\theoremstyle{definition}
\newtheorem{definition}[theorem]{Definition}

\newtheorem{example}[theorem]{Example}
\newtheorem{observation}[theorem]{Observation}

\newcommand{\N}{\mathbb{N}}

\newcommand{\R}{\mathbb{R}}

\newcommand{\lk}{\operatorname{lk}}
\newcommand{\st}{\operatorname{st}}
\newcommand{\dlk}{\lk^\downarrow\!}

\newcommand{\dst}{\st^\downarrow\!}
\newcommand{\dflk}{\lk^\downarrow_\partial}
\newcommand{\dclk}{\lk^\downarrow_\delta}

\newcommand{\hasse}{\mathcal{H}}
\newcommand{\morse}{\mathcal{M}}
\newcommand{\genmorse}{\mathcal{GM}}

\newcommand{\defeq}{\mathbin{\vcentcolon =}}

\numberwithin{equation}{section}

\begin{document}

\title{Higher connectivity of the Morse complex}
\date{\today}


\author[N.~A.~Scoville]{Nicholas A. Scoville}
\address{Department of Mathematics and Computer Science, Ursinus College, Collegeville, PA 19426}
\email{nscoville@ursinus.edu}

\author[M.~C.~B.~Zaremsky]{Matthew C.~B.~Zaremsky}
\address{Department of Mathematics and Statistics, University at Albany (SUNY), Albany, NY 12222}
\email{mzaremsky@albany.edu}

\begin{abstract}
The Morse complex $\morse(\Delta)$ of a finite simplicial complex $\Delta$ is the complex of all gradient vector fields on $\Delta$. In this paper we study higher connectivity properties of $\morse(\Delta)$. For example, we prove that $\morse(\Delta)$ gets arbitrarily highly connected as the maximum degree of a vertex of $\Delta$ goes to $\infty$, and for $\Delta$ a graph additionally as the number of edges goes to $\infty$. We also classify precisely when $\morse(\Delta)$ is connected or simply connected. Our main tool is Bestvina--Brady Morse theory, applied to a ``generalized Morse complex.''
\end{abstract}

\maketitle
\thispagestyle{empty}

\section*{Introduction}

The Morse complex $\morse(\Delta)$ of a finite simplicial complex $\Delta$ is the simplicial complex of all gradient vector fields on $\Delta$. See Section~\ref{sec:morse_cpx} for a more detailed definition. The Morse complex $\morse(\Delta)$ has several important properties. For example, two connected simplicial complexes are isomorphic if and only if their Morse complexes are isomorphic \cite{capitelli17}. Additionally, outside a few sporadic cases, for connected $\Delta$ the group of automorphisms of $\morse(\Delta)$ is isomorphic to that of $\Delta$ \cite{lin21}. The Morse complex may be viewed as a discrete analog of the space of gradient vector fields on a manifold; see, e.g., \cite{Palis82}.

The homotopy type of $\morse(\Delta)$ is only known for a handful of examples of $\Delta$, and in general it is difficult to compute. In this paper, we relax the question to just asking how highly connected $\morse(\Delta)$ is (meaning up to what bound the homotopy groups vanish). Our first main result is the following:

\begin{main:main_1_ground_morse}
If $\Delta$ has a vertex with degree $d$ in $\Delta^{(1)}$ then $\morse(\Delta)$ is $(d-2)$-connected.
\end{main:main_1_ground_morse}

For example this holds if $\dim(\Delta)\ge d$. It is harder to obtain good higher connectivity bounds when the dimension of $\Delta$ is small and vertices of $\Delta$ have small degrees,  but for certain situations we can. First we focus on the case when $\dim(\Delta)=1$, i.e., $\Delta$ is a graph $\Gamma$. Here we are able to use Bestvina--Brady Morse theory, applied to the so called generalized Morse complex $\genmorse(\Gamma)$, to find higher connectivity bounds for $\morse(\Gamma)$. Let $d(\Gamma)$ be the maximum degree of a vertex in the Hasse diagram. Our main result for graphs is:

\begin{main:main_graphs}
The Morse complex $\morse(\Gamma)$ is $\left(\left\lceil\frac{|E(\Gamma)|}{d(\Gamma)}\right\rceil-2\right)$-connected.
\end{main:main_graphs}

Combining Theorem~\ref{thrm:main_graphs} with Theorem~\ref{thrm:1_ground_morse} quickly shows that, as the number of edges of $\Gamma$ goes to $\infty$, $\morse(\Gamma)$ becomes arbitrarily highly connected (see Corollary~\ref{cor:number_of_edges} for a precise statement). We conjecture that a similar result holds regardless of $\dim(\Delta)$ (Conjecture~\ref{conj:higher}).

Our last main result is a classification of precisely which $\Delta$ have connected and simply connected Morse complexes. Here we assume $\Delta$ has no isolated vertices just to make the statement cleaner (isolated vertices can be deleted without affecting $\morse(\Delta)$).

\begin{main:main_classification}
Suppose $\Delta$ has no isolated vertices. The Morse complex $\morse(\Delta)$ is connected if and only if $\Delta$ is not an edge, and is simply connected if and only if $\Delta$ is none of: an edge, a disjoint union of two edges, a path with three edges, a $3$-cycle, or a $2$-simplex.
\end{main:main_classification}

This paper is organized as follows. In Section~\ref{sec:morse_cpx} we set up the Morse complex $\morse(\Delta)$ and generalized Morse complex $\genmorse(\Delta)$. In Section~\ref{sec:first_hi_conn} we prove Theorem~\ref{thrm:1_ground_morse}. In Section~\ref{sec:bb} we discuss Bestvina--Brady discrete Morse theory and how to apply it to $\genmorse(\Delta)$. In Section~\ref{sec:graphs} we focus on the situation for graphs and prove Theorem~\ref{thrm:main_graphs}. Finally, in Section~\ref{sec:hi_dim} we discuss the situation for $\dim(\Delta)>1$ and prove Theorem~\ref{thrm:classification}.

\subsection*{Acknowledgments} We are grateful to the organizers of the 2019 Union College Mathematics Conference, where the idea for this project originated. We are also very grateful to the anonymous referees, for many helpful comments and in particular for suggestions that led us to proving Theorem~\ref{thrm:1_ground_morse} and its ramifications, which greatly strengthened our results. The second named author is supported by grant \#635763 from the Simons Foundation.

\section{The Morse complex}\label{sec:morse_cpx}

Let $\Delta$ be a finite abstract simplicial complex. We will abuse notation and also write $\Delta$ for the geometric realization of $\Delta$. If $\sigma$ is a $p$-dimensional simplex in $\Delta$, we may write $\sigma^{(p)}$ to indicate the dimension. A \emph{primitive discrete vector field} on $\Delta$ is a pair $(\sigma^{(p)},\tau^{(p+1)})$ for $\sigma<\tau$. A \emph{discrete vector field} $V$ on $\Delta$ is a collection of primitive discrete vector fields
\[
V=\{(\sigma_0,\tau_0),\dots,(\sigma_k,\tau_k)\}
\]
such that each simplex of $\Delta$ is in at most one pair $(\sigma_i,\tau_i)$. If the two simplices in $(\sigma,\tau)$ are distinct from the two simplices in $(\sigma',\tau')$, call the primitive discrete vector fields $(\sigma,\tau)$ and $(\sigma',\tau')$ \emph{compatible}; in particular a discrete vector field is a set of pairwise compatible primitive discrete vector fields.

The \emph{Hasse diagram} of $\Delta$ is the simple graph $\hasse(\Delta)$ with a vertex for each (non-empty) simplex of $\Delta$ and an edge between any pair of simplices such that one is a codimension-$1$ face of the other. In particular the primitive discrete vector fields on $\Delta$ are in one-to-one correspondence with the edges of $\hasse(\Delta)$. Also, the discrete vector fields on $\Delta$ are in one-to-one correspondence with the \emph{matchings}, i.e., the collections of pairwise disjoint edges, on $\hasse(\Delta)$. We will sometimes equivocate between a discrete vector field on $\Delta$ and its corresponding matching on $\hasse(\Delta)$.

\begin{definition}[Generalized Morse complex]
The \emph{generalized Morse complex} $\genmorse(\Delta)$ of $\Delta$ is the simplicial complex whose vertices are the primitive discrete vector fields on $\Delta$, with a finite collection of vertices spanning a simplex whenever the primitive discrete vector fields are pairwise compatible. Said another way, the simplices of $\genmorse(\Delta)$ are the discrete vector fields on $\Delta$, with face relation given by inclusion.
\end{definition}

Note that $\genmorse(\Delta)$ is a flag complex, i.e., if a finite collection of vertices pairwise span edges then they span a simplex, which makes it comparatively easy to analyze. Viewed in terms of matchings on $\hasse(\Delta)$, $\genmorse(\Delta)$ is precisely the \emph{matching complex} of $\hasse(\Delta)$, i.e., the simplicial complex of matchings with face relation given by inclusion. Matching complexes of graphs are well studied; see \cite{bayer20} for an especially extensive list of references.

The Morse complex $\morse(\Delta)$ of $\Delta$, introduced by Chari and Joswig in \cite{chari05}, is the subcomplex of $\genmorse(\Delta)$ consisting of all discrete vector fields arising from a Forman discrete Morse function, or equivalently all acyclic discrete vector fields. To define all this, we need some setup, which we draw mostly from \cite[Section~2.2]{scoville19}. Given a discrete vector field $V$ on $\Delta$, a \emph{$V$-path} is a sequence of simplices
\[
\sigma_0^{(p)}, \tau_0^{(p+1)}, \sigma_1^{(p)}, \tau_1^{(p+1)}, \sigma_2^{(p)}, \dots , \tau_{m-1}^{(p+1)}, \sigma_m^{(p)}
\]
such that for each $0\le i\le m-1$, $(\sigma_i,\tau_i)\in V$ and $\tau_i>\sigma_{i+1}\ne \sigma_i$. Such a $V$-path is \emph{non-trivial} if $m>0$, and \emph{closed} if $\sigma_m=\sigma_0$. A closed non-trivial $V$-path is called a \emph{$V$-cycle}. If there exist no $V$-cycles, call $V$ \emph{acyclic}. A $V$-cycle is \emph{simple} if $\sigma_0,\dots,\sigma_{m-1}$ are pairwise distinct and $\tau_0,\dots,\tau_{m-1}$ are pairwise distinct. We will identify $V$-cycles up to cyclic permutation, e.g., we consider $\sigma_0,\tau_0,\sigma_1,\dots,\tau_{m-1},\sigma_0$ to be the same cycle as $\sigma_1,\tau_1,\sigma_2,\dots,\tau_{m-1},\sigma_0,\tau_0,\sigma_1$, and so forth.

Every acyclic discrete vector field on $\Delta$ is the gradient vector field of a Forman discrete Morse function on $\Delta$. A \emph{Forman discrete Morse function} on $\Delta$ (developed by Forman in \cite{forman98}) is a function $h\colon \Delta\to\R$ such that for every $\sigma^{(p)}$, there is at most one $\tau^{(p+1)}>\sigma^{(p)}$ with $h(\tau)\le h(\sigma)$, and for every $\tau^{(p+1)}$ there is at most one $\sigma^{(p)}<\tau^{(p+1)}$ with $h(\sigma)\ge h(\tau)$. The \emph{gradient vector field} of $h$ is the discrete vector field whose primitive vector fields are all the $(\sigma^{(p)},\tau^{(p+1)})$ with $h(\sigma)\ge h(\tau)$. A discrete vector field is the gradient vector field of some Forman discrete Morse function if and only if it is acyclic \cite[Theorem 2.51]{scoville19}.

\begin{definition}[Morse complex]
The subcomplex $\morse(\Delta)$ of $\genmorse(\Delta)$ consisting of all acyclic $V$ is the \emph{Morse complex} of $\Delta$.
\end{definition}

Note that any subset of an acyclic discrete vector field is itself acyclic, so $\morse(\Delta)$ really is a subcomplex. We should remark that the term ``Morse complex'' also means a certain algebraic chain complex obtained from an acyclic matching, e.g., see \cite[Definition 11.23]{kozlov08}, but in this paper ``Morse complex'' will always mean $\morse(\Delta)$.

The following observation will be important later when relating $\morse(\Delta)$ and $\genmorse(\Delta)$.

\begin{observation}[1-skeleton]\label{obs:1-skel}
The $1$-skeleton of $\morse(\Delta)$ coincides with that of $\genmorse(\Delta)$.
\end{observation}

\begin{proof}
Since $\Delta$ is simplicial, fewer than three compatible primitive discrete vector fields cannot form a cycle.
\end{proof}

Let us discuss two examples that are instructive and will be specifically relevant later.

\begin{example}\label{ex:3_cycle}
Let $\Delta=C_3$ be the $3$-cycle, i.e, the cyclic graph with $3$ vertices. See Figure~\ref{fig:3_cycle} for drawings of $\hasse(C_3)$, $\genmorse(C_3)$, and $\morse(C_3)$. We see that $\genmorse(C_3)\simeq S^1\vee S^1$ and $\morse(C_3)\simeq S^1\vee S^1\vee S^1\vee S^1$ (this computation of $\morse(C_3)$ agrees with Kozlov's computation in \cite[Proposition~5.2]{kozlov99}). In particular neither $\morse(C_3)$ nor $\genmorse(C_3)$ is simply connected.
\end{example}

\begin{figure}[htb]
\centering
\begin{tikzpicture}[scale=.7]

\coordinate (z) at (0,0);
\coordinate (y) at (2,0);
\coordinate (x) at (4,0);
\coordinate (w) at (3,1.732);
\coordinate (v) at (2,3.464);
\coordinate (u) at (1,1.732);

\draw[line width=1] (z) -- (x) -- (v) -- (z);

\filldraw (z) circle (1.5pt);
\filldraw (y) circle (1.5pt);
\filldraw (x) circle (1.5pt);
\filldraw (w) circle (1.5pt);
\filldraw (v) circle (1.5pt);
\filldraw (u) circle (1.5pt);

\node at ($.5*(z)+.5*(u)+(-.25,.15)$) {$a$};
\node at ($.5*(u)+.5*(v)+(-.25,.15)$) {$b$};
\node at ($.5*(v)+.5*(w)+(.25,.15)$) {$c$};
\node at ($.5*(w)+.5*(x)+(.25,.15)$) {$d$};
\node at ($.5*(x)+.5*(y)+(0,-.3)$) {$e$};
\node at ($.5*(y)+.5*(z)+(0,-.3)$) {$f$};

\begin{scope}[xshift=6cm,yshift=4cm]
\coordinate (a) at (0,0);
\coordinate (c) at (1.732,-1);
\coordinate (e) at (0,-2);
\coordinate (f) at (3.732,-1);
\coordinate (d) at (5.464,0);
\coordinate (b) at (5.464,-2);

\filldraw[lightgray] (a) -- (c) -- (e);
\filldraw[lightgray] (f) -- (d) -- (b);
\draw[line width=1] (c) -- (e) -- (a) -- (c) -- (f) -- (d) -- (b) -- (f);
\draw[line width=1] (a) -- (d)   (e) -- (b);

\filldraw (a) circle (2pt);
\filldraw (b) circle (2pt);
\filldraw (c) circle (2pt);
\filldraw (d) circle (2pt);
\filldraw (e) circle (2pt);
\filldraw (f) circle (2pt);

\node at ($(a)-(.3,0)$) {$a$};
\node at ($(e)-(.3,0)$) {$e$};
\node at ($(c)+(0,.3)$) {$c$};
\node at ($(f)+(0,.4)$) {$f$};
\node at ($(d)+(.3,0)$) {$d$};
\node at ($(b)+(.3,0)$) {$b$};
\end{scope}

\begin{scope}[xshift=6cm,yshift=1cm]
\coordinate (a) at (0,0);
\coordinate (c) at (1.732,-1);
\coordinate (e) at (0,-2);
\coordinate (f) at (3.732,-1);
\coordinate (d) at (5.464,0);
\coordinate (b) at (5.464,-2);

\draw[line width=1] (c) -- (e) -- (a) -- (c) -- (f) -- (d) -- (b) -- (f);
\draw[line width=1] (a) -- (d)   (e) -- (b);

\filldraw (a) circle (2pt);
\filldraw (b) circle (2pt);
\filldraw (c) circle (2pt);
\filldraw (d) circle (2pt);
\filldraw (e) circle (2pt);
\filldraw (f) circle (2pt);

\node at ($(a)-(.3,0)$) {$a$};
\node at ($(e)-(.3,0)$) {$e$};
\node at ($(c)+(0,.3)$) {$c$};
\node at ($(f)+(0,.4)$) {$f$};
\node at ($(d)+(.3,0)$) {$d$};
\node at ($(b)+(.3,0)$) {$b$};
\end{scope}

\end{tikzpicture}
\caption{The Hasse diagram $\hasse(C_3)$ (left), the generalized Morse complex $\genmorse(C_3)$ (top), and the Morse complex $\morse(C_3)$ (bottom).}
\label{fig:3_cycle}
\end{figure}

\begin{example}\label{ex:2_spx}
Let $\Delta=\Delta^2$ be the $2$-simplex. See Figure~\ref{fig:2_spx} for drawings of $\hasse(\Delta^2)$, $\genmorse(\Delta^2)$, and $\morse(\Delta^2)$. We see that $\genmorse(\Delta^2)\simeq S^1\vee S^1$ and $\morse(\Delta^2)\simeq S^1\vee S^1\vee S^1\vee S^1$ (this computation of $\morse(\Delta^2)$ agrees with Chari and Joswig's \cite[Proposition~5.1]{chari05}). In particular neither $\morse(\Delta^2)$ nor $\genmorse(\Delta^2)$ is simply connected.
\end{example}

\begin{figure}[H]
\centering
\begin{tikzpicture}[scale=.8]

\coordinate (z) at (0,0);
\coordinate (y) at (2,0);
\coordinate (x) at (4,0);
\coordinate (w) at (3,1.732);
\coordinate (v) at (2,3.464);
\coordinate (u) at (1,1.732);
\coordinate (t) at (2,1.155);

\draw[line width=1] (z) -- (x) -- (v) -- (z);
\draw[line width=1] (y) -- (t) -- (u)   (t) -- (w);

\filldraw (z) circle (1.5pt);
\filldraw (y) circle (1.5pt);
\filldraw (x) circle (1.5pt);
\filldraw (w) circle (1.5pt);
\filldraw (v) circle (1.5pt);
\filldraw (u) circle (1.5pt);
\filldraw (t) circle (1.5pt);

\node at ($.5*(z)+.5*(u)+(-.25,.15)$) {$a$};
\node at ($.5*(u)+.5*(v)+(-.25,.15)$) {$b$};
\node at ($.5*(v)+.5*(w)+(.25,.15)$) {$c$};
\node at ($.5*(w)+.5*(x)+(.25,.15)$) {$d$};
\node at ($.5*(x)+.5*(y)+(0,-.3)$) {$e$};
\node at ($.5*(y)+.5*(z)+(0,-.3)$) {$f$};
\node at ($.5*(u)+.5*(t)+(.2,.15)$) {$g$};
\node at ($.5*(w)+.5*(t)+(-.2,.17)$) {$h$};
\node at ($.5*(y)+.5*(t)+(.2,0)$) {$i$};

\begin{scope}[xshift=6cm,yshift=5cm]
\coordinate (a) at (0,0);
\coordinate (c) at (1.732,-1);
\coordinate (e) at (0,-2);
\coordinate (f) at (3.732,-1);
\coordinate (d) at (5.464,0);
\coordinate (b) at (5.464,-2);
\coordinate (g) at ($.5*(c)+.5*(f)+(0,.2)$);
\coordinate (h) at ($.5*(e)+.5*(b)+(-1,-1)$);
\coordinate (i) at ($.5*(a)+.5*(d)+(1,1)$);

\filldraw[lightgray] (a) -- (c) -- (e);
\draw[line width=1] (c) -- (e) -- (a) -- (c);
\filldraw[lightgray] (f) -- (d) -- (b);
\draw[line width=1] (f) -- (d) -- (b) -- (f);

\filldraw[lightgray] (f) -- (g) -- (d);
\draw[line width=1] (f) -- (g) -- (d) -- (f);
\filldraw[lightgray] (f) -- (g) -- (c);
\draw[line width=1] (f) -- (g) -- (c) -- (f);
\filldraw[lightgray] (e) -- (g) -- (c);
\draw[line width=1] (e) -- (g) -- (c) -- (e);

\filldraw[lightgray] (f) -- (b) -- (h);
\draw[line width=1] (f) -- (b) -- (h) -- (f);
\filldraw[lightgray] (e) -- (b) -- (h);
\draw[line width=1] (e) -- (b) -- (h) -- (e);
\filldraw[lightgray] (e) -- (a) -- (h);
\draw[line width=1] (e) -- (a) -- (h) -- (e);

\filldraw[lightgray] (c) -- (a) -- (i);
\draw[line width=1] (c) -- (a) -- (i) -- (c);
\filldraw[lightgray] (d) -- (a) -- (i);
\draw[line width=1] (d) -- (a) -- (i) -- (d);
\filldraw[lightgray] (d) -- (b) -- (i);
\draw[line width=1] (d) -- (b) -- (i) -- (d);

\filldraw (a) circle (2pt);
\filldraw (b) circle (2pt);
\filldraw (c) circle (2pt);
\filldraw (d) circle (2pt);
\filldraw (e) circle (2pt);
\filldraw (f) circle (2pt);
\filldraw (g) circle (2pt);
\filldraw (h) circle (2pt);
\filldraw (i) circle (2pt);

\node at ($(a)-(.3,0)$) {$a$};
\node at ($(e)-(.3,0)$) {$e$};
\node at ($(c)+(0,.3)$) {$c$};
\node at ($(f)+(0,-.4)$) {$f$};
\node at ($(d)+(.3,0)$) {$d$};
\node at ($(b)+(.3,0)$) {$b$};
\node at ($(g)+(0,.3)$) {$g$};
\node at ($(h)+(0,-.3)$) {$h$};
\node at ($(i)+(0,.3)$) {$i$};
\end{scope}

\begin{scope}[xshift=6cm,yshift=0cm]
\coordinate (a) at (0,0);
\coordinate (c) at (1.732,-1);
\coordinate (e) at (0,-2);
\coordinate (f) at (3.732,-1);
\coordinate (d) at (5.464,0);
\coordinate (b) at (5.464,-2);
\coordinate (g) at ($.5*(c)+.5*(f)+(0,.2)$);
\coordinate (h) at ($.5*(e)+.5*(b)+(-1,-1)$);
\coordinate (i) at ($.5*(a)+.5*(d)+(1,1)$);

\draw[line width=1] (c) -- (e) -- (a) -- (c);
\draw[line width=1] (f) -- (d) -- (b) -- (f);

\filldraw[lightgray] (f) -- (g) -- (d);
\draw[line width=1] (f) -- (g) -- (d) -- (f);
\filldraw[lightgray] (f) -- (g) -- (c);
\draw[line width=1] (f) -- (g) -- (c) -- (f);
\filldraw[lightgray] (e) -- (g) -- (c);
\draw[line width=1] (e) -- (g) -- (c) -- (e);

\filldraw[lightgray] (f) -- (b) -- (h);
\draw[line width=1] (f) -- (b) -- (h) -- (f);
\filldraw[lightgray] (e) -- (b) -- (h);
\draw[line width=1] (e) -- (b) -- (h) -- (e);
\filldraw[lightgray] (e) -- (a) -- (h);
\draw[line width=1] (e) -- (a) -- (h) -- (e);

\filldraw[lightgray] (c) -- (a) -- (i);
\draw[line width=1] (c) -- (a) -- (i) -- (c);
\filldraw[lightgray] (d) -- (a) -- (i);
\draw[line width=1] (d) -- (a) -- (i) -- (d);
\filldraw[lightgray] (d) -- (b) -- (i);
\draw[line width=1] (d) -- (b) -- (i) -- (d);

\filldraw (a) circle (2pt);
\filldraw (b) circle (2pt);
\filldraw (c) circle (2pt);
\filldraw (d) circle (2pt);
\filldraw (e) circle (2pt);
\filldraw (f) circle (2pt);
\filldraw (g) circle (2pt);
\filldraw (h) circle (2pt);
\filldraw (i) circle (2pt);

\node at ($(a)-(.3,0)$) {$a$};
\node at ($(e)-(.3,0)$) {$e$};
\node at ($(c)+(0,.3)$) {$c$};
\node at ($(f)+(0,-.4)$) {$f$};
\node at ($(d)+(.3,0)$) {$d$};
\node at ($(b)+(.3,0)$) {$b$};
\node at ($(g)+(0,.3)$) {$g$};
\node at ($(h)+(0,-.3)$) {$h$};
\node at ($(i)+(0,.3)$) {$i$};
\end{scope}

\end{tikzpicture}
\caption{The Hasse diagram $\hasse(\Delta^2)$ (left), the generalized Morse complex $\genmorse(\Delta^2)$ (top), and the Morse complex $\morse(\Delta^2)$ (bottom).}
\label{fig:2_spx}
\end{figure}

It will become necessary later to consider the following generalization of $\genmorse(\Delta)$ and $\morse(\Delta)$, in which certain simplices are ``illegal'' and cannot be used. Specifically, this will be needed in the proof of Proposition~\ref{prop:main_gen_graphs} to get an inductive argument to work.

\begin{definition}[Relative (generalized) Morse complex]
Let $\Omega$ be a subset of the set of simplices of $\Delta$. The \emph{relative generalized Morse complex} $\genmorse(\Delta,\Omega)$ is the full subcomplex of $\genmorse(\Delta)$ spanned by those vertices, i.e., primitive discrete vector fields $(\sigma,\tau)$, such that $\sigma,\tau\not\in \Omega$. The \emph{relative Morse complex} $\morse(\Delta,\Omega)$ is the subcomplex $\morse(\Delta)\cap \genmorse(\Delta,\Omega)$.
\end{definition}

We can also phrase things using $\hasse(\Delta)$.

\begin{definition}[Relative Hasse diagram]
The \emph{relative Hasse diagram} $\hasse(\Delta,\Omega)$ is the induced subgraph of $\hasse(\Delta)$ with vertex set given by all simplices of $\Delta$ not in $\Omega$.
\end{definition}

If we view $\genmorse(\Delta)$ as the matching complex of $\hasse(\Delta)$, then clearly $\genmorse(\Delta,\Omega)$ is the matching complex of $\hasse(\Delta,\Omega)$.

\section{First results on higher connectivity}\label{sec:first_hi_conn}

In this section we establish some higher connectivity bounds for the various complexes in question. In subsequent sections we will use Bestvina--Brady Morse theory to obtain more sophisticated higher connectivity bounds in certain cases. First we focus on the relative generalized Morse complex $\genmorse(\Delta,\Omega)$. We will use the ``Belk--Forrest groundedness trick,'' introduced by Belk and Forrest in \cite{belk19}.

\begin{definition}[Ground, grounded]
Call a simplex in a simplicial complex an \emph{$r$-ground} if every vertex of the complex is adjacent to all but at most $r$ vertices of the simplex. The complex is \emph{$(k,r)$-grounded} if it admits a $k$-simplex that is an $r$-ground.
\end{definition}

\begin{theorem}[Groundedness trick]\cite[Theorem~4.9]{belk19}\label{thrm:ground}
Every $(k,r)$-grounded flag complex is $\left (\left \lceil\frac{k+1}{r}\right \rceil-2\right )$-connected.
\end{theorem}

Note that in \cite[Theorem~4.9]{belk19} the complex is assumed to be finite and $k,r$ are assumed to be at least $1$, but this is not necessary: see, e.g., \cite[Remark~4.12]{skipper19}. Also note that in these references the bound is written $\left \lfloor\frac{k}{r}\right \rfloor-1$, but this equals $\left \lceil\frac{k+1}{r}\right \rceil-2$, and this form will be notationally convenient for us later.

In $\genmorse(\Delta,\Omega)$ it is clear that every $k$-simplex is a $(k,2)$-ground. This is because any primitive discrete vector field only ``uses'' two simplices of $\Delta$, and so can fail to be compatible with at most two vertices of a given simplex. In particular this shows:

\begin{observation}\label{obs:ground}
If $\genmorse(\Delta,\Omega)$ contains a $k$-simplex then it is $\left(\left \lceil\frac{k+1}{2}\right \rceil-2\right)$-connected.\qed
\end{observation}

Note that this only works because $\genmorse(\Delta,\Omega)$ is a flag complex, and in particular Theorem~\ref{thrm:ground} does not apply to $\morse(\Delta,\Omega)$.

This next result will be useful later when using Bestvina--Brady Morse theory and inductive arguments. Let $h(\Delta,\Omega)$ be the number of edges in $\hasse(\Delta,\Omega)$, and let $d(\Delta,\Omega)$ be the maximum degree of a vertex in $\hasse(\Delta,\Omega)$.

\begin{proposition}\label{prop:genmorse_hi_conn}
The complex $\genmorse(\Delta,\Omega)$ is $\left(\left\lceil\frac{h(\Delta,\Omega)}{2d(\Delta,\Omega)}\right\rceil-2\right)$-connected.
\end{proposition}

\begin{proof}
We first claim that $\genmorse(\Delta,\Omega)$ contains a simplex of dimension $\left\lceil\frac{h(\Delta,\Omega)}{d(\Delta,\Omega)}\right\rceil-1$. A $k$-simplex in $\genmorse(\Delta,\Omega)$ consists of $k+1$ pairwise disjoint edges in $\hasse(\Delta,\Omega)$, so we need to show that $\hasse(\Delta,\Omega)$ admits $\left\lceil\frac{h(\Delta,\Omega)}{d(\Delta,\Omega)}\right\rceil$ pairwise disjoint edges. Since $\hasse(\Delta,\Omega)$ is a simple bipartite graph, by K\H{o}nig's Theorem it suffices to show that every vertex cover of $\hasse(\Delta,\Omega)$ has at least $\left\lceil\frac{h(\Delta,\Omega)}{d(\Delta,\Omega)}\right\rceil$ vertices. (Here a \emph{vertex cover} is a subset $S$ of the vertex set such that every edge is incident to at least one element of $S$.) Indeed for any graph $\Theta$, if $S$ is a vertex cover of $\Theta$ then
\[
|S|\max\{\deg(v)\mid v\in V(\Theta)\}  \ge \sum_{v\in S}\deg(v) \ge |E(\Theta)| \text{,}
\]
 and we have $|E(\hasse(\Delta,\Omega))|=h(\Delta,\Omega)$ and $\max\{\deg(v)\mid v\in V(\hasse(\Delta,\Omega))\}=d(\Delta,\Omega)$, so this follows.

Now set $k=\left\lceil\frac{h(\Delta,\Omega)}{d(\Delta,\Omega)}\right\rceil-1$, so we have shown that $\genmorse(\Delta,\Omega)$ contains a simplex of dimension $k$. Then $\genmorse(\Delta,\Omega)$ is $(k,2)$-grounded, and is flag, so Theorem~\ref{thrm:ground} implies that $\genmorse(\Delta,\Omega)$ is $(\lceil\frac{k+1}{2}\rceil-2)$-connected, hence $\left(\left\lceil\frac{h(\Delta,\Omega)}{2d(\Delta,\Omega)}\right\rceil-2\right)$-connected.
\end{proof}

Note that in Proposition~\ref{prop:genmorse_hi_conn} we obtain a higher connectivity bound that is better when the maximal degree $d(\Delta,\Omega)$ is small. We can also find a better, higher connectivity bound when the maximal degree is large. We will only need this in the $\Omega=\emptyset$ case (since we will not need to use Morse theory or induction later), so for simplicity we will only phrase it in that case, but one could state an analog when $\Omega\ne\emptyset$.

\begin{lemma}\label{lem:1_ground}
Suppose $\Delta$ has a vertex that has degree $d$ in $\Delta^{(1)}$. Then $\genmorse(\Delta)$ is $(d-1,1)$-grounded, and hence $(d-2)$-connected.
\end{lemma}

\begin{proof}
Let $u$ be a vertex of degree $d$ in $\Delta^{(1)}$. Let $V$ be the $(d-1)$-simplex $\{(v_1,e_1),\dots,(v_d,e_d)\}$ with each $e_i$ an edge incident to $u$ and each $v_i$ the endpoint of $e_i$ not equal to $u$. We claim that $V$ is a $(d-1,1)$-ground. Indeed, for $i\ne j$ we have that $v_i$ is not incident to $e_j$, so if $(\sigma,\tau)$ is an arbitrary vertex in $\genmorse(\Delta)$ then $\{\sigma,\tau\}$ can intersect $\{v_i,e_i\}$ for at most one $i$. This shows that $\genmorse(\Delta)$ is $(d-1,1)$-grounded, and since it is flag Theorem~\ref{thrm:ground} says it is $(d-2)$-connected.
\end{proof}

Of course the actual goal of this paper is to find higher connectivity results for $\morse(\Delta)$. Even though $\morse(\Delta)$ is not flag, and so the groundedness trick does not apply, we can still prove the analog of Lemma~\ref{lem:1_ground} for $\morse(\Delta)$ using a more complicated argument. First let us record an easy lemma that will be important in many arguments that follow.

\begin{lemma}\label{lem:one_cycle}
Let $(v^{(0)},e^{(1)})$ be a primitive discrete vector field in a discrete vector field $V$ on $\Delta$. Then $(v,e)$ lies in at most one simple $V$-cycle.
\end{lemma}

\begin{proof}
Let $v'$ be the endpoint of $e$ not equal to $v$. Assume that $(v,e)$ lies in a simple $V$-cycle. Then $v'$ must be matched in $V$ to some edge $e'$. Since $v'$ cannot be matched in $V$ to more than one edge, $e'$ is the unique edge with $(v',e')\in V$. Hence every $V$-cycle containing $(v,e)$ also contains $(v',e')$. Repeating this argument, we see that if $(v,e)$ lies in a simple $V$-cycle this simple $V$-cycle is unique.
\end{proof}

Note that the analog of Lemma~\ref{lem:one_cycle} is not true for $(\sigma^{(p)},\tau^{(p+1)})$ with $p>0$, since then $\tau$ can have more than two codimension-$1$ faces.

Recall that the \emph{star} $\st_X(\sigma)$ of a simplex $\sigma$ in a simplicial complex $X$ is the subcomplex of all simplices containing $\sigma$ along with their faces. Let us say two simplices of $X$ are \emph{joinable (in $X$)} if they lie in a common simplex in $X$, or equivalently if they lie in each other's stars.

\begin{theorem}\label{thrm:1_ground_morse}
If $\Delta$ has a vertex with degree $d$ in $\Delta^{(1)}$ then $\morse(\Delta)$ is $(d-2)$-connected.
\end{theorem}

\begin{proof}
As in the proof of Lemma~\ref{lem:1_ground}, let $u$ be a vertex of degree $d$ in $\Delta^{(1)}$. Let $V$ be the $(d-1)$-simplex $\{(v_1,e_1),\dots,(v_d,e_d)\}$ with each $e_i$ an edge incident to $u$ and each $v_i$ the endpoint of $e_i$ not equal to $u$. Note that $V$ is acyclic, hence a simplex in $\morse(\Delta)$, since every $e_i$ has $u$ as its endpoint different than $v_i$. We first claim that the union of stars
\[
\bigcup\limits_{i=1}^d \st_{\morse(\Delta)}(v_i,e_i)
\]
is contractible. Since stars are contractible, and these stars all intersect, e.g., they all contain $V$, it suffices by the Nerve Lemma \cite[Lemma~1.2]{bjoerner94} to show that the intersection of the stars of any subcollection of the $(v_i,e_i)$ is contractible. We claim that for any face $V'$ of $V$, we have
\[
\bigcap\limits_{(v_i,e_i)\in V'}\st_{\morse(\Delta)}(v_i,e_i)=\st_{\morse(\Delta)}(V') \text{.}
\]
This will prove the claim since $\st_{\morse(\Delta)}(V')$ is contractible. The reverse inclusion holds trivially, so we need to prove the forward inclusion. Let $W$ be a simplex in $\morse(\Delta)$ that is joinable in $\morse(\Delta)$ to $(v_i,e_i)$ for each vertex $(v_i,e_i)$ in $V'$. Since $\genmorse(\Delta)$ is flag, this implies $W$ and $V'$ span a simplex $W\cup V'$ in $\genmorse(\Delta)$, and we need to show that $W\cup V'$ is acyclic. Any simple cycle in $W\cup V'$ can contain at most one $(v_i,e_i)$, since every $e_i$ has $u$ as its endpoint, $u\neq v_i$. Since $W$ is joinable in $\morse(\Delta)$ to each vertex of $V'$, no such cycles can exist. We conclude that $W$ lies in $\st_{\morse(\Delta)}(V')$, so $\bigcup\limits_{i=1}^d \st_{\morse(\Delta)}(v_i,e_i)$ is contractible.

Now we claim that this union contains the $(d-2)$-skeleton of $\morse(\Delta)$. Let $U$ be a $(d-2)$-simplex in $\morse(\Delta)$. Since $\genmorse(\Delta)$ is a $(d-1,1)$-grounded flag complex, with $(d-1,1)$-ground $V$ by the proof of Lemma~\ref{lem:1_ground}, every vertex of $U$ is adjacent (in $\genmorse(\Delta)$) to all but at most one vertex of $V$. Since $V$ has $d$ vertices and $U$ has $d-1$ vertices, this implies there exists a vertex $(v_i,e_i)$ of $V$ such that every vertex of $U$ is compatible with $(v_i,e_i)$. Since $\genmorse(\Delta)$ is flag, this shows $U$ lies in $\st_{\genmorse(\Delta)}(v_i,e_i)$. As a first case, suppose $U$ lies in more than one such star, say $\st_{\genmorse(\Delta)}(v_1,e_1)$ and $\st_{\genmorse(\Delta)}(v_2,e_2)$. If $U$ does not lie in $\st_{\morse(\Delta)}(v_1,e_1)$ then the discrete vector field $U\sqcup \{(v_1,e_1)\}$ has a cycle. Since $U$ has no cycles, this means there is a cycle in $U\sqcup \{(v_1,e_1)\}$ containing $(v_1,e_1)$. This cycle necessarily contains a primitive discrete vector field of the form $(u,e_j)$ for some $j$. Since $(u,e_j)$ can lie in at most one simple cycle in the discrete vector field $U\cup \{(v_1,e_1),(v_2,e_2)\}$, by Lemma~\ref{lem:one_cycle}, and since $(v_1,e_1)$ and $(v_2,e_2)$ cannot lie in a common cycle, we conclude that $(v_2,e_2)$ lies in no cycles in $U\cup \{(v_1,e_1),(v_2,e_2)\}$. In particular $(v_2,e_2)$ lies in no cycles in $U\cup \{(v_2,e_2)\}$, so $U$ is in $\st_{\morse(\Delta)}(v_2,e_2)$, which finishes this case.

For the other case, suppose $U$ only lies in one $\st_{\genmorse(\Delta)}(v_i,e_i)$, say without loss of generality in $\st_{\genmorse(\Delta)}(v_1,e_1)$. Then for every $2\le i\le d$, $U$ has a vertex $P_i$ that is a primitive discrete vector field incompatible with $(v_i,e_i)$, i.e., $P_i$ contains either $v_i$ or $e_i$ but not both. Since no $v_i$ is incident to any $e_j$ for $i\ne j$, the function $i\mapsto P_i$ must be injective, so $P_2,\dots,P_d$ are precisely the $d-1$ vertices of $U$. If $P_i\ne (u,e_i)$ for any $i$, then $U\sqcup\{(v_1,e_1)\}$ cannot contain a cycle. Hence suppose, without loss of generality, that $P_2=(u,e_2)$. Then for each $i\ge 3$, $P_i$ either contains $v_i$ or is of the form $(e_i,f_i)$ for some $2$-simplex $f_i$. In particular no $P_i$ contains $v_2$. Hence no cycle in $U\sqcup\{(v_1,e_1)\}$ can contain $(u,e_2)$, which implies no cycle in $U\sqcup\{(v_1,e_1)\}$ can contain $(v_1,e_1)$, which implies there are no cycles. We conclude $U$ is in $\st_{\morse(\Delta)}(v_1,e_1)$, which finishes this case.

We have shown that $\morse(\Delta)^{(d-2)}$ lies in a contractible subcomplex of $\morse(\Delta)$. Hence the inclusion $\morse(\Delta)^{(d-2)}\to \morse(\Delta)$ induces the trivial map in all homotopy groups. We also know this map induces a surjection in all $\pi_k$ for $k\le d-2$, so $\morse(\Delta)$ is $(d-2)$-connected.
\end{proof}

It seems much more difficult to prove the analog of Proposition~\ref{prop:genmorse_hi_conn} for $\morse(\Delta)$, but we conjecture that it holds. Let us record this here (with $\Omega=\emptyset$ for simplicity):

\begin{conjecture}\label{conj:higher}
The Morse complex $\morse(\Delta)$ is $\left(\left\lceil\frac{h(\Delta)}{2d(\Delta)}\right\rceil-2\right)$-connected. In particular if $h(\Delta)>2m\cdot d(\Delta)$ then $\morse(\Delta)$ is $(m-1)$-connected.
\end{conjecture}

In the following sections we will use Bestvina--Brady Morse theory to prove this conjecture in the case when $\dim(\Delta)=1$, i.e., for graphs, and for the special cases $m=1,2$ regardless of $\dim(\Delta)$.

\section{Bestvina--Brady Morse theory}\label{sec:bb}

An important tool we will use now is Bestvina--Brady discrete Morse theory. This is related to Forman's discrete Morse theory, and in fact can be viewed as a generalization of it, as explained in \cite{zaremsky}. For our purposes the definition of a Bestvina--Brady discrete Morse function is as follows. (This is a special case of the situation considered in \cite{zaremsky}.)

\begin{definition}\label{def:bb}
Let $X$ be a simplicial complex and $\phi,\psi\colon X^{(0)}\to\R$ two functions such that for any adjacent vertices $x,y\in X^{(0)}$ we have $(\phi,\psi)(x)\ne (\phi,\psi)(y)$. Extend $\phi$ and $\psi$ to maps $X\to\R$ by extending affinely to each simplex. Then we call
\[
(\phi,\psi)\colon X\to \R\times\R
\]
a \emph{Bestvina--Brady discrete Morse function} provided the following holds: for any infinite sequence $x_1,x_2,\dots$ of vertices such that for each $i$, $x_i$ is adjacent to $x_{i+1}$ and $(\phi,\psi)(x_i)>(\phi,\psi)(x_{i+1})$ lexicographically, the set $\{\phi(x_1),\phi(x_2),\dots\}$ has no lower bound in $\R$.
\end{definition}

Note that if $X$ is finite, as it will be in our forthcoming applications, then this condition about infinite sequences holds vacuously, but for now we will continue working in full generality.

Definition~\ref{def:bb} is a bit unwieldy, but we will only need the following special case:

\begin{example}\label{ex:bcsd}
Let $X=Y'$ be the barycentric subdivision of a simplicial complex $Y$, so the vertices of $X$ are the simplices of $Y$ and adjacency in $X$ is determined by incidence in $Y$. Let $\phi\colon X^{(0)}\to\R$ be any function. Let $\dim\colon X^{(0)}\to\R$ be the function sending $\sigma$ (viewed as a vertex of $X$) to $\dim(\sigma)$ (viewed as a simplex of $Y$). If $Y$ is finite dimensional and $\phi(X^{(0)})\subseteq \R$ is closed and discrete (for example this holds if $X$ is finite), then $(\phi,-\dim)\colon X\to \R$ is a Bestvina--Brady discrete Morse function. Indeed, adjacent vertices of $X$ have different $\dim$ values (hence different $(\phi,-\dim)$ values), and the finite dimensionality of $Y$ plus the fact that $\phi(X^{(0)})$ is closed and discrete ensures that the last condition of Definition~\ref{def:bb} is satisfied.
\end{example}

Given a Bestvina--Brady discrete Morse function $(\phi,\psi)\colon X\to\R$, we can deduce topological properties of the sublevel complexes $X^{\phi\le t}$ by analyzing topological properties of the descending links of vertices. Here the \emph{sublevel complex} $X^{\phi\le t}$ for $t\in\R\cup\{\infty\}$ is the full subcomplex of $X$ spanned by vertices $x$ with $\phi(x)\le t$. The \emph{descending link} $\dlk x$ of a vertex $x$ is the space of directions out of $x$ in which $(\phi,\psi)$ decreases in the lexicographic order. More rigorously, since $\phi$ and $\psi$ are affine on simplices and not simultaneously constant on edges, the lexicographic pair $(\phi,\psi)$ achieves its maximum value on a given simplex at a unique vertex of the simplex, called its \emph{top}. The \emph{descending star} $\dst x$ is the subcomplex of $X$ consisting of all simplices with top $x$, and their faces. Then $\dlk x$ is the link of $x$ in $\dst x$.

The claim that an understanding of descending links leads to an understanding of sublevel complexes is made rigorous by the following Morse Lemma. This is essentially \cite[Corollary~2.6]{bestvina97}, and is more precisely spelled out in this form in, e.g., \cite[Corollary~1.11]{zaremsky}.

\begin{lemma}[Morse Lemma]\label{lem:morse}
Let $(\phi,\psi)\colon X\to\R$ be a Bestvina--Brady discrete Morse function on a simplicial complex $X$. Let $t<s$ in $\R\cup\{\infty\}$. If $\dlk x$ is $(n-1)$-connected for all vertices $x$ with $t<\phi(x)\le s$ then the inclusion $X^{\phi\le t}\to X^{\phi\le s}$ induces an isomorphism in $\pi_k$ for all $k\le n-1$, and an epimorphism in $\pi_n$.
\end{lemma}

Let us return to the special case from Example~\ref{ex:bcsd}, so $X=Y'$ for $Y$ finite dimensional, and $\phi\colon X\to\R$ is closed and discrete on vertices. Given a vertex $\sigma$ in $X$ (i.e., a simplex in $Y$), there are two types of vertex in $\dlk\sigma$: we can either have a face $\sigma^\vee <\sigma$ with $\phi(\sigma^\vee)<\phi(\sigma)$, or a coface $\sigma^\wedge >\sigma$ with $\phi(\sigma^\wedge)\le\phi(\sigma)$. This is because $-\dim$ goes up when passing to faces and down when passing to cofaces. Since every face of $\sigma$ is a face of every coface of $\sigma$, the descending link $\dlk\sigma$ decomposes as join
\[
\dlk\sigma=\dflk\sigma * \dclk\sigma \text{,}
\]
where $\dflk\sigma$, the \emph{descending face link}, is spanned by all $\sigma^\vee <\sigma$ with $\phi(\sigma^\vee)<\phi(\sigma)$, and $\dclk\sigma$, the \emph{descending coface link}, is spanned by all $\sigma^\wedge >\sigma$ with $\phi(\sigma^\wedge)\le\phi(\sigma)$. For example if at least one of $\dflk\sigma$ or $\dclk\sigma$ is contractible, so is $\dlk\sigma$. More generally, an understanding of the topology of $\dflk\sigma$ and $\dclk\sigma$ yields an understanding of the topology of $\dlk\sigma$.

\subsection{Applying Bestvina--Brady Morse theory to the relative generalized Morse complex}\label{sec:bb_on_gen_morse}

Now we will apply Bestvina--Brady Morse theory to $\genmorse(\Delta,\Omega)$. The broad strokes of this strategy are inspired by the Morse theoretic approach in \cite[Proposition~3.6]{bux16} to higher connectivity properties of the matching complex of a complete graph. Let $X=\genmorse(\Delta,\Omega)'$, and let $\phi\colon X^{(0)}\to\N\cup\{0\}$ be the function sending $V$ to the number of simple $V$-cycles (since $\Delta$ is finite, any $V$ only has finitely many simple $V$-cycles). In particular $X^{\phi\le 0}=\morse(\Delta,\Omega)'$, so if we can understand $\dlk V$ for all $V$ with $\phi(V)>0$, using the Bestvina--Brady discrete Morse function $(\phi,-\dim)$ as in Example~\ref{ex:bcsd}, then the Morse Lemma will tell us information about $\morse(\Delta,\Omega)'\cong \morse(\Delta,\Omega)$.

Let us inspect the descending face link.

\begin{lemma}[Descending face link, case 1]\label{lem:dflk_cible}
Let $V\in X^{(0)}$ with $\phi(V)>0$, so $V$ is a discrete vector field on $\Delta$ (avoiding $\Omega$) with at least one $V$-cycle. If there exists a primitive discrete vector field in $V$ that is not contained in any $V$-cycle, then $\dflk V$ is contractible.
\end{lemma}

\begin{proof}
Say $V=\{(\sigma_0,\tau_0),\dots,(\sigma_k,\tau_k)\}$, and say without loss of generality that $(\sigma_0,\tau_0)$ is not contained in any $V$-cycle. Now let $W$ be any vertex of $\dflk V$, so $W$ is a simplex of $\genmorse(\Delta,\Omega)$ with $W<V$ and $\phi(W)<\phi(V)$. Then $\phi(W\cup\{(\sigma_0,\tau_0)\})=\phi(W)<\phi(V)$, so $W\cup\{(\sigma_0,\tau_0)\}\in \dflk V$. Since $W\le W\cup\{(\sigma_0,\tau_0)\}\ge \{(\sigma_0,\tau_0)\}$, \cite[Section~1.5]{quillen78} says $\dflk V$ is contractible.
\end{proof}

\begin{lemma}[Descending face link, case 2]\label{lem:dflk_all}
Let $V\in X^{(0)}$ with $\phi(V)>0$, say $V$ is a $k$-simplex of $\genmorse(\Delta,\Omega)$. If every primitive discrete vector field in $V$ is contained in a $V$-cycle, then $\dflk V$ is homeomorphic to $S^{k-1}$.
\end{lemma}

\begin{proof}
The hypothesis ensures that $\phi(W)<\phi(V)$ for every proper face $W<V$, i.e., removing any part of $V$ eliminates at least one $V$-cycle (note that removing part of $V$ cannot create new cycles, so these are in fact equivalent). Hence $\dflk V$ is homeomorphic to the boundary of $V$ (viewed as a simplex in $\genmorse(\Delta,\Omega)$), so homeomorphic to $S^{k-1}$.
\end{proof}

At this point we know that the descending link of a $k$-simplex $V$ with $\phi(V)>0$ is either contractible, or else is the join of $S^{k-1}$ with $\dclk V$ (so the $k$-fold suspension of $\dclk V$). It remains to analyze $\dclk V$. In Section~\ref{sec:graphs} we will discuss the case when $\dim(\Delta)=1$, where it turns out we can fully analyze $\dclk V$. Then in Section~\ref{sec:hi_dim} we will consider arbitrary $\Delta$, where at least we will be able to tell when $\dclk V$ is non-empty.

\section{Graphs}\label{sec:graphs}

In the special case when $\dim(\Delta)=1$, i.e., $\Delta=\Gamma$ is a graph, the descending coface link of those $V$ satisfying the hypotheses of Lemma~\ref{lem:dflk_all} can be related to a ``smaller'' Morse complex (see Proposition~\ref{prop:dclk_iso}), which allows for inductive arguments. Throughout this section $\Gamma$ denotes a finite graph, and $\Omega$ is a subset of the set of simplices of $\Gamma$. To us ``graph'' will always mean a $1$-dimensional simplicial complex, often called a ``simple graph''.

\begin{proposition}[Modeling the descending coface link]\label{prop:dclk_iso}
Let $V$ be a $k$-simplex in $\genmorse(\Gamma,\Omega)$ such that every primitive discrete vector field in $V$ lies in a $V$-cycle. Let $\Upsilon$ be the set of simplices of $\Gamma$ used by $V$. Then $\dclk V$ in $X$ is isomorphic to $\morse(\Gamma,\Omega\cup\Upsilon)'$.
\end{proposition}

\begin{proof}
Define a simplicial map $\psi\colon \dclk V\to \morse(\Gamma,\Omega\cup\Upsilon)'$ as follows. A vertex of $\dclk V$ is a discrete vector field on $\Gamma$ (avoiding $\Omega$) of the form $V\sqcup W$ for non-trivial $W$ such that any $V\sqcup W$-cycle is a $V$-cycle. In particular $W$ is acyclic, and so $W\in \morse(\Gamma,\Omega\cup\Upsilon)$. Setting $\psi\colon (V\sqcup W)\mapsto W$ gives a well defined map on the level of vertices. If $V\sqcup W < V\sqcup W'$ then $W<W'$, so this extends to a simplicial map $\psi\colon \dclk V\to \morse(\Gamma,\Omega\cup\Upsilon)'$. Now we have to show $\psi$ is bijective. It is clearly injective, since $W=W'$ implies $V\sqcup W = V\sqcup W'$. It is also clear that as long as $\psi$ is surjective on vertices, it will be surjective. To see it is surjective on vertices, let $W$ be a vertex in $\morse(\Gamma,\Omega\cup\Upsilon)'$, and we have to show that any $V\sqcup W$-cycle is a $V$-cycle, since then $V\sqcup W$ will be a vertex in $\dclk V$. Note that for any primitive discrete vector field $(v,e)$ in $V$, our assumptions say that $(v,e)$ lies in a $V$-cycle. Since any $V$-cycle is also a $V\sqcup W$-cycle, Lemma~\ref{lem:one_cycle} says $(v,e)$ cannot lie in any $V\sqcup W$-cycles other than this one. Hence any $V\sqcup W$-cycle that contains a primitive discrete vector field in $V$ must be completely contained in $V$. Finally, note that any non-trivial $V\sqcup W$-cycle cannot be fully contained in $W$ since $W$ is acyclic. We conclude that any $V\sqcup W$-cycle is a $V$-cycle.
\end{proof}

We reiterate that the analog of Lemma~\ref{lem:one_cycle} is not true for simplicial complexes of dimension greater than $1$, so this proof does not work outside the graph case.

\begin{proposition}\label{prop:main_gen_graphs}
The Morse complex $\morse(\Gamma,\Omega)$ is $\left(\left\lceil\frac{h(\Gamma,\Omega)}{2d(\Gamma,\Omega)}\right\rceil-2\right)$-connected.
\end{proposition}

\begin{proof}
We induct on $h(\Gamma,\Omega)$. The base case is that $\morse(\Gamma,\Omega)$ is non-empty once $h(\Gamma,\Omega)>0$, which is clear. Now assume $h(\Gamma,\Omega)>2d(\Gamma,\Omega)$. By the Morse Lemma~\ref{lem:morse} and Proposition~\ref{prop:genmorse_hi_conn}, it suffices to show that for $V$ a $k$-simplex in $\genmorse(\Gamma,\Omega)$ with $\phi(V)>0$, the descending link $\dlk V$ is $\left(\left\lceil\frac{h(\Gamma,\Omega)}{2d(\Gamma,\Omega)}\right\rceil-2\right)$-connected. If there exists a primitive discrete vector field in $V$ that is not contained in any $V$-cycle, then $\dflk V$ (and hence $\dlk V$) is contractible by Lemma~\ref{lem:dflk_cible}. Now assume every primitive discrete vector field in $V$ is contained in a $V$-cycle. Then by Lemma~\ref{lem:dflk_all}, $\dflk V \cong S^{k-1}$, and by Proposition~\ref{prop:dclk_iso}, $\dclk V\cong \morse(\Gamma,\Omega\cup\Upsilon)$, where $\Upsilon$ is the set of simplices used in $V$. Since $\dlk V = \dflk V * \dclk V$, it now suffices to show that $\morse(\Gamma,\Omega\cup\Upsilon)$ is $\left(\left\lceil\frac{h(\Gamma,\Omega)}{2d(\Gamma,\Omega)}\right\rceil-k-2\right)$-connected. By induction $\morse(\Gamma,\Omega\cup\Upsilon)$ is $\left(\left\lceil\frac{h(\Gamma,\Omega\cup\Upsilon)}{2d(\Gamma,\Omega\cup\Upsilon)}\right\rceil-2\right)$-connected. Note that $h(\Gamma,\Omega\cup\Upsilon)\ge h(\Gamma,\Omega)-((k+1)(d(\Gamma,\Omega)+1)-1)$. This is because removing $k+1$ edges and their endpoints and all their incident edges would normally remove at most $(k+1)(d(\Gamma,\Omega)+1)$ total edges from $\hasse(\Gamma,\Omega)$ (here we use the fact that a vertex of $\hasse(\Gamma,\Omega)$ representing an edge of $\Gamma$ has degree at most $2$), but since $V$ has at least one cycle we know that we only removed at most $(k+1)(d(\Gamma,\Omega)+1)-1$ total edges. Also, $d(\Gamma,\Omega\cup\Upsilon)\le d(\Gamma,\Omega)$, so $\morse(\Gamma,\Omega\cup\Upsilon)$ is $\left(\left\lceil\frac{h(\Gamma,\Omega)-((k+1)(d(\Gamma,\Omega)+1)-1)}{2d(\Gamma,\Omega)}\right\rceil-2\right)$-connected.  Since $\phi(V)>0$ we know $k\ge 2$ by Observation~\ref{obs:1-skel}. Also, if $d(\Gamma,\Omega)=1$ then $\morse(\Gamma,\Omega)=\genmorse(\Gamma,\Omega)$ and we are done, so we can assume $d(\Gamma,\Omega)\ge 2$. Putting all this together we compute:
\begin{align*}
\left(\left\lceil\frac{h(\Gamma,\Omega)-((k+1)(d(\Gamma,\Omega)+1)-1)}{2d(\Gamma,\Omega)}\right\rceil-2\right) &= \\
\left(\left\lceil\frac{h(\Gamma,\Omega)-(kd(\Gamma,\Omega)+k+d(\Gamma,\Omega))}{2d(\Gamma,\Omega)}\right\rceil-2\right) &\ge \\
\left(\left\lceil\frac{h(\Gamma,\Omega)-2kd(\Gamma,\Omega)}{2d(\Gamma,\Omega)}\right\rceil-2\right) &= \\
\left(\left\lceil\frac{h(\Gamma,\Omega)}{2d(\Gamma,\Omega)}\right\rfloor-k-2\right)
\end{align*}
and we are done.
\end{proof}

In the special case where $\Omega=\emptyset$, we can now draw conclusions about $\morse(\Gamma)$. Let us write $h(\Gamma)=h(\Gamma,\emptyset)$ and $d(\Gamma)=d(\Gamma,\emptyset)$, so $h(\Gamma)=2|E(\Gamma)|$ and $d(\Delta)$ is the maximum degree of a vertex in the Hasse diagram (which is usually the same as the maximum degree of a vertex in $\Gamma$, unless every vertex of $\Gamma$ has degree $0$ or $1$).

\begin{theorem}\label{thrm:main_graphs}
The Morse complex $\morse(\Gamma)$ is $\left(\left\lceil\frac{|E(\Gamma)|}{d(\Gamma)}\right\rceil-2\right)$-connected. In particular $\morse(\Gamma)$ is connected once $|E(\Gamma)|>d(\Gamma)$, simply connected once $|E(\Gamma)|>2d(\Gamma)$, and $(m-1)$-connected once $|E(\Gamma)|>m\cdot d(\Gamma)$.
\end{theorem}

\begin{proof}
By Proposition~\ref{prop:main_gen_graphs} $\morse(\Gamma)$ is $\left(\left\lceil\frac{2|E(\Gamma)|}{2d(\Gamma)}\right\rceil-2\right)$-connected, i.e., $\left(\left\lceil\frac{|E(\Gamma)|}{d(\Gamma)}\right\rceil-2\right)$-connected.
\end{proof}

This proves Conjecture~\ref{conj:higher} when $\dim(\Delta)=1$. Combining this with Theorem~\ref{thrm:1_ground_morse} we can obtain a higher connectivity bound that only depends on $|E(\Gamma)|$. Let $\eta(\Gamma)\defeq \left\lceil\sqrt{|E(\Gamma)|}\right \rceil$.

\begin{corollary}\label{cor:number_of_edges}
The Morse complex $\morse(\Gamma)$ is $(\eta(\Gamma)-2)$-connected. In particular $\morse(\Gamma)$ is connected once $|E(\Gamma)|>1$, simply connected once $|E(\Gamma)|>4$, and $(m-1)$-connected once $|E(\Gamma)|>m^2$.
\end{corollary}

\begin{proof}
If $\Gamma$ has no vertices of degree $d(\Gamma)$ then $\Gamma$ is a disjoint union of edges, and so is $(|E(\Gamma)|-2)$-connected, hence $(\eta(\Gamma)-2)$-connected. Now assume $\Gamma$ has a vertex of degree $d(\Gamma)$. By Theorem~\ref{thrm:1_ground_morse}, $\morse(\Gamma)$ is $(d(\Gamma)-2)$-connected. If $d(\Gamma)\ge \eta(\Gamma)$ then we are done, so assume $d(\Gamma)\le \eta(\Gamma)-1$. Then $\left\lceil\frac{|E(\Gamma)|}{d(\Gamma)}\right\rceil \ge \left\lceil\frac{|E(\Gamma)|}{\eta(\Gamma)-1}\right\rceil \ge \left\lceil\frac{|E(\Gamma)|}{\sqrt{|E(\Gamma)|}}\right\rceil = \eta(\Gamma)$, and so we are done by Theorem~\ref{thrm:main_graphs}.
\end{proof}

\subsection{Examples}\label{sec:graph_examples}

Now we discuss a couple of examples. First let us discuss an example where the homotopy type of $\morse(\Gamma)$ is already known, namely when $\Gamma$ is a complete graph. This example will show that, while our results are powerful in that they apply to any $\Gamma$, they do not necessarily yield optimal bounds.

\begin{example}
Let $K_n$ be the complete graph on $n$ vertices, so $|E(K_n)|=\binom{n}{2}$. By Corollary~\ref{cor:number_of_edges} $\morse(K_n)$ is $(m-1)$-connected once $\binom{n}{2}>m^2$, i.e., once $n>(1+\sqrt{1+8m^2})/2$. For example it is connected once $n>2$ and simply connected once $n>3$. Kozlov computed the homotopy type $\morse(K_n)$, namely $\morse(K_n)$ is homotopy equivalent to a wedge of spheres of dimension $n-2$ \cite[Theorem~3.1]{kozlov99}, so in fact $\morse(K_n)$ is already $(m-1)$-connected once $n>m+1$. This shows our bounds are not always optimal.
\end{example}

As a remark, Kozlov also computed the homotopy type of $\morse(C_n)$ \cite[Proposition~5.2]{kozlov99} for $C_n$ the $n$-cycle graph. Since $|E(C_n)|=n$ it is easy to compare our higher connectivity bounds to the actual higher connectivity, and again we see our bounds are not optimal.

Now we discuss an example where the homotopy type of $\morse(\Gamma)$ is (to the best of our knowledge) not known, namely when $\Gamma$ is complete bipartite, and compute what our results reveal.

\begin{example}
Let $K_{p,q}$ be the complete bipartite graph with $p$ vertices of one type and $q$ vertices of the other type, so $|E(K_{p,q})|=pq$. By Corollary~\ref{cor:number_of_edges} $\morse(K_{p,q})$ is $(m-1)$-connected once $pq>m^2$. For example it is connected once $pq>1$ and simply connected once $pq>4$. Later in Theorem~\ref{thrm:classification} we will see that actually it is also simply connected once $pq>1$, i.e., every $\morse(K_{p,q})$ is simply connected except $\morse(K_{1,1})$, which is not even connected.
\end{example}

\section{Higher dimensional $\Delta$}\label{sec:hi_dim}

Now we consider arbitrary dimensional $\Delta$, and prove some results about higher connectivity properties of $\morse(\Delta)$. First we observe that $\morse(\Delta)$ gets arbitrarily highly connected as $\dim(\Delta)$ goes to $\infty$.

\begin{theorem}\label{thrm:hi_dim}
Suppose that $\Delta$ contains a $k$-simplex. Then $\morse(\Delta)$ is $(k-2)$-connected.
\end{theorem}

\begin{proof}
Since $\Delta$ contains a $k$-simplex, $\Delta^{(1)}$ contains a vertex of degree $k$. The result is now immediate from Theorem~\ref{thrm:1_ground_morse}.
\end{proof}

Our next goal is to completely classify when $\morse(\Delta)$ is connected and simply connected. To do this we will first prove that $\morse(\Delta)$ is (simply) connected if and only if $\genmorse(\Delta)$ is, for which we will use Bestvina--Brady Morse theory applied to $X=\genmorse(\Delta)'$, as in Section~\ref{sec:graphs}. The key is that, even without a full understanding of $\dclk V$ in the $\dim(\Delta)>1$ case, we will only need to care that $\dclk V$ is non-empty. Bestvina--Brady Morse theory is probably a more powerful tool than necessary to relate $\morse(\Delta)$ to $\genmorse(\Delta)$ in this way, but it makes for an elegant argument.

\begin{lemma}[Descending link simply connected]\label{lem:dlk_cimp_conn}
Assume $\Delta$ is not a $2$-simplex or a $3$-cycle. Then for any $V\in X^{(0)}$ with $\phi(V)>0$, either $\dflk V$ is simply connected or $\dflk V$ is connected and $\dclk V$ is non-empty. In particular, the descending link $\dlk V$ is always simply connected.
\end{lemma}

\begin{proof}
Say $V$ is a $k$-simplex of $\genmorse(\Delta)$. The fact that $\phi(V)>0$ implies $k\ge 2$, by Observation~\ref{obs:1-skel}. We see from Lemmas~\ref{lem:dflk_cible} and~\ref{lem:dflk_all} that $\dflk V$ is either contractible or homeomorphic to $S^{k-1}$. If $k\ge 3$ this is simply connected. Now we have to prove that if $k=2$ then $\dclk V$ is non-empty.

Say $V=\{(\sigma_0,\tau_0),(\sigma_1,\tau_1),(\sigma_2,\tau_2)\}$. Since $\phi(V)>0$ we know that $\sigma_0,\tau_0,\sigma_1,\tau_1,\sigma_2,\tau_2,\sigma_0$ is a $V$-cycle, so the $\sigma_i$ all have the same dimension, say $p$, and the $\tau_i$ all have dimension $p+1$. First suppose $p>0$. Then $\dim(\tau_0)\ge 2$, so we can choose a $1$-face $e<\tau_0$ and $0$-face $v<e$ such that $e$ is disjoint from $\sigma_0$ and $\sigma_1$. In particular $V\sqcup \{(v,e)\}$ is a discrete vector field, and it is clear that $\phi(V\sqcup \{(v,e)\})=\phi(V)$, so $\dclk V\ne\emptyset$.

Now suppose $p=0$, so the $\sigma_i$ are vertices and the $\tau_i$ are edges. If $\Delta$ contains an edge $e$ not equal to any $\tau_i$ then $e$ must have at least one vertex $v$ not equal to any $\sigma_i$. In this case $V\sqcup \{(v,e)\}$ is a discrete vector field, and it is clear that $\phi(V\sqcup \{(v,e)\})=\phi(V)$, so $\dclk V\ne\emptyset$. Finally, suppose $\Delta$ does not contain any edges besides the $\tau_i$. Since isolated vertices do not contribute to the Morse complex we can assume $\Delta$ has none, so the only options are that $\Delta$ equals a $2$-simplex or a $3$-cycle, but we ruled these out.
\end{proof}

\begin{corollary}\label{cor:genmorse_same_as_morse}
We have that $\morse(\Delta)$ is connected if and only if $\genmorse(\Delta)$ is connected, and $\morse(\Delta)$ is simply connected if and only if $\genmorse(\Delta)$ is simply connected.
\end{corollary}

\begin{proof}
First note that $\morse(\Delta)^{(1)}=\genmorse(\Delta)^{(1)}$ by Observation~\ref{obs:1-skel}, so the connectivity result is true. Now we prove the simple connectivity result. If $\Delta$ is a $2$-simplex or a $3$-cycle, then Examples~\ref{ex:3_cycle} and~\ref{ex:2_spx} show that the result holds. Now assume $\Delta$ is neither of these. By Lemma~\ref{lem:dlk_cimp_conn} the descending link of every $V$ with $\phi(V)>0$ is simply connected. Thus by the Morse Lemma~\ref{lem:morse}, the inclusion $\morse(\Delta)\to\genmorse(\Delta)$ induces an isomorphism in $\pi_1$.
\end{proof}

Now we can completely classify when $\morse(\Delta)$ is connected and simply connected.

\begin{theorem}\label{thrm:classification}
Suppose $\Delta$ has no isolated vertices. The Morse complex $\morse(\Delta)$ is connected if and only if $\Delta$ is not an edge, and is simply connected if and only if $\Delta$ is none of: an edge, a disjoint union of two edges, a path with three edges, a $3$-cycle, or a $2$-simplex.
\end{theorem}

\begin{proof}
First we prove the connectivity statement. If $\Delta^{(1)}$ has a vertex with degree more than $1$ then $\genmorse(\Delta)$, and hence $\morse(\Delta)$, is connected by Lemma~\ref{lem:1_ground} and Corollary~\ref{cor:genmorse_same_as_morse}. Now assume $\Delta^{(1)}$ has no vertices with degree more than $1$, so $\Delta$ is a disjoint union of edges. If $\Delta$ has at least two edges (or, for trivial reasons, zero edges) then it is easy to check that $\morse(\Delta)$ is connected. If $\Delta$ has one edge then $\morse(\Delta)=S^0$ is not connected.

Now we prove the simple connectivity statement. If $\Delta^{(1)}$ has a vertex with degree more than $2$ then $\genmorse(\Delta)$, and hence $\morse(\Delta)$, is simply connected by Lemma~\ref{lem:1_ground} and Corollary~\ref{cor:genmorse_same_as_morse}. Now assume $\Delta^{(1)}$ has no vertices with degree more than $2$. Then $\Delta$ is a disjoint union of some number of $2$-simplices, cycle graphs, and path graphs. If $\Delta$ has more than one connected component, and is not a disjoint union of two edges, then $\morse(\Delta)$ is a join of at least two non-empty complexes, at least one of which is connected, and so $\morse(\Delta)$ is simply connected. If $\Delta$ is a disjoint union of two edges then $\morse(\Delta)\simeq S^1$ is not simply connected. Now assume $\Delta$ is connected. If it is a $2$-simplex then $\morse(\Delta)$ is not simply connected (Example~\ref{ex:2_spx}). If $\Delta$ is an $n$-cycle then $\morse(\Delta)$ is simply connected unless $n=3$ \cite[Proposition~5.2]{kozlov99}. If $\Delta$ is a path with $n$ edges then $\genmorse(\Delta)$ is the matching complex of a path with $2n$ edges, which is easily seen to be simply connected unless $n$ is $1$ or $3$, so the same is true of $\morse(\Delta)$ by Corollary~\ref{cor:genmorse_same_as_morse}.
\end{proof}

A consequence of Theorem~\ref{thrm:classification} is that we can now verify the connectivity and simple connectivity cases of Conjecture~\ref{conj:higher}. Let us use the notation $h(\Delta)=h(\Delta,\emptyset)$ and $d(\Delta)=d(\Delta,\emptyset)$ as before, so $h(\Delta)$ is the number of edges in the Hasse diagram of $\Delta$ and $d(\Delta)$ is the maximum degree of a vertex in the Hasse diagram.

\begin{corollary}\label{cor:simp_conn}
If $h(\Delta)>2d(\Delta)$ then $\morse(\Delta)$ is connected. If $h(\Delta)>4d(\Delta)$ then $\morse(\Delta)$ is simply connected.
\end{corollary}

\begin{proof}
Since isolated vertices do not contribute to $h(\Delta)$, $d(\Delta)$, or $\morse(\Delta)$, we can assume there are none. If $\morse(\Delta)$ is not connected then $\Delta$ is an edge by Theorem~\ref{thrm:classification}, so $h(\Delta)=2<4=2d(\Delta)$. If $\morse(\Delta)$ is not simply connected then $\Delta$ is either an edge, a disjoint union of two edges, a path with three edges, a $3$-cycle, or a $2$-simplex. In all these cases one can compute that $h(\Delta)\le 4d(\Delta)$.
\end{proof}

\bibliographystyle{alpha}
\newcommand{\etalchar}[1]{$^{#1}$}

\end{document}